\documentclass[11pt,a4paper]{amsart}
\usepackage[english]{babel}
\usepackage{amssymb,xspace}
\usepackage{amstext}

\theoremstyle{plain}
\usepackage{amsbsy,amssymb,amsfonts,latexsym}

\newcommand{\D}{\mathbb D}
\newcommand{\R}{\mathbb R}
\newcommand{\C}{\mathbb C}

\newcommand{\bC}{\mathbf C}
\newcommand{\bD}{\mathbf D}
\newcommand{\T}{\partial \D}

\newcommand{\BH}{{\mathcal B}(H) }
\newcommand{\BX}{{\mathcal B}(X) }
\newcommand{\Ker}{\operatorname{Ker}}
\newcommand{\Ran}{\operatorname{Ran}}

\newcommand{\re}{\operatorname{Re}}
\newcommand{\dist}{\operatorname{dist}}
\newcommand{\pocc}{\mathcal{S}}
\newcommand{\ep}{\varepsilon}

\newcommand{\ld}{\lambda}
\newcommand{\sm}{\sigma}
\newcommand\ps[2]{\left\langle #1,#2\right\rangle}

\numberwithin{equation}{section}

{\theoremstyle{definition}\newtheorem{num}{}[section]} %numerotation

\newcommand{\beq}{\begin{equation}}
\newcommand{\eeq}{\end{equation}}

\newtheorem{theorem}[num]{Theorem}

\newtheorem{lemma}[num]{Lemma}
\newtheorem{proposition}[num]{Proposition}

{\theoremstyle{definition}}
{\theoremstyle{definition}\newtheorem{example}[num]{Example}}

{\theoremstyle{definition}\newtheorem{definition}[num]{Definition}}
{\theoremstyle{definition}\newtheorem{remark}[num]{Remark}}

\title[The method of alternating projections]{The rate of convergence in the method of alternating projections}\thanks{The first two named authors have been partially supported by ANR Projet Blanc DYNOP. The third named author was supported by grant No. 201/09/0473 of GA \v CR and IAA100190903 of GA AV. {\bf To appear in} : St. Petersburg Math. J. (translation from Russian of Algebra and Analysis) {\bf 22}(2010), no. 5. Announced in C. R. Math. Acad. Sci. Paris  {\bf 348}(2010), no. 1-2, 53--56.}
\author{Catalin Badea}
\address{Laboratoire Paul Painlev\' e,
Universit\'e Lille 1, CNRS UMR 8524,
F-59655 Villeneuve d'Ascq, France
}
\email{badea@math.univ-lille1.fr}
\author{Sophie Grivaux}
\address{Laboratoire Paul Painlev\' e, CNRS UMR 8524,
Universit\'e Lille 1,
F-59655 Villeneuve d'Ascq, France
}
\email{grivaux@math.univ-lille1.fr}
\author{Vladimir M\"uller}
\address{Institute of Mathematics AV CR, Zitna 25, 115 67 Prague 1, Czech Republic}
\email{muller@math.cas.cz}

\begin{document}

 \begin{abstract}
A generalization of the cosine of the Friedrichs angle between two subspaces to a parameter associated to several closed subspaces of a Hilbert space is given. This parameter is used to analyze the rate of convergence in the von Neumann-Halperin method of cyclic alternating projections. General dichotomy theorems are proved, in the Hilbert or Banach space situation, providing conditions under which the alternative QUC/ASC (quick uniform convergence versus arbitrarily slow convergence) holds. Several meanings for ASC are proposed.
\end{abstract}
\maketitle

\section{Introduction}
Throughout the paper $H$ is a complex Hilbert space. For a closed
linear subspace $S$ of $H$ we denote by $S^{\perp}$ its orthogonal
complement in $H$, and by $P_S$ the orthogonal projection of $H$
onto $S$. In this paper $N$ denotes a fixed positive integer greater or 
equal than $2$.

\subsection{The method of alternating projections} It was proved by J.~von~Neumann \cite[p. 475]{vN} that for two closed subspaces $M_1$ and $M_2$ of $H$, with intersection $M = M_1\cap M_2$, the following convergence result holds:
\beq \label{eq:11}
\lim_{n \to \infty}\|(P_{M_2}P_{M_1})^n(x) -  P_M(x)\| = 0 \quad (x \in H).
\eeq
 Using the notation $T = P_{M_2}P_{M_1}$, von Neumann's result says that the iterates $T^n$ of $T$ are strongly convergent to $T^{\infty} = P_M$. The method of constructing the iterates of $T$ by alternately projecting onto one subspace and then the other is called the \emph{method of alternating projections}. This algorithm, and its variations, occur in several fields, pure or applied. We refer to \cite[Chapter 9]{deutsch:book} as a source for more information.

A generalization of von Neumann's result to $N$ closed subspaces $M_1, \dots , M_N$ with intersection $M = M_1\cap M_2 \cdots \cap M_N$ was proved by Halperin \cite{halperin}: for each $x \in H$ we have
\beq \label{eq:12}
\lim_{n \to \infty}\|(P_{M_N} \cdots P_{M_2}P_{M_1})^n(x) -  P_M(x)\| = 0 .
\eeq
The algorithm provided by Halperin's result will be called in this paper the \emph{method of cyclic alternating projections}.

A Banach space extension of Halperin's result was proved by Bruck and Reich \cite{bruck/reich}: if $X$ is a \emph{uniformly convex} Banach space and $P_j$, $1\le j \le N$, are $N$ \emph{norm one projections} in $\BX$, then the iterates of $T = P_{N} \cdots P_{2}P_{1}$ are strongly convergent. The strong limit $T^{\infty}$ is a projection of norm one onto the intersection of the ranges of $P_j$. The same result holds \cite{BaLy} if $X$ is \emph{uniformly smooth} and each projection $P_j$ is of norm one. It also holds \cite{BaLy} if $X$ is a reflexive (complex) Banach space and each projection $P_j$ is hermitian (that is, with real numerical range). We refer to \cite{BaLy} and the references therein for other Banach space results of this type.

An interesting extension of the method of cyclic alternating projections is the\emph{ method of random alternating projections}. Let $P_j$, $1\le j \le N$, be $N$ orthogonal projections in $\BH$, $M = \cap_{j=1}^N \Ran(P_j)$, and let $(i_k)_{k\ge 1}$ be a sequence from $\{1,2, \dots, N\}$ (random samples). The method of random alternating projections asks about the convergence of the sequence $(x_n)_{n\ge 0}$ given by $x_0 = x$, $x_n = P_{i_n}x_{n-1}$. It is an open problem to know whether $(x_n)_{n\ge 0}$ is always convergent in the topology of $H$. The convergence of $(x_n)_{n\ge 0}$ in the \emph{weak} topology has been proved by Amemiya and Ando \cite{amemiya/ando}. If each $j$ between $1$ and $N$ occurs infinitely many times in the sequence of random samples, then the weak limit of $(x_n)_{n\ge 0}$ is $P_Mx$. We refer to \cite{DKR,sakai,KKM} for results related to this problem.

\subsection{The rate of convergence} It is important for applications to know how fast
the algorithm given by the method of alternating projections, or
its variations, converge. For $N=2$ a quite complete description
of the rate of convergence is known, it terms of the notion of
\emph{angle} of subspaces.
\begin{definition}(Friedrichs angle) Let $M_1$ and $M_2$ be two closed subspaces of the Hilbert space $H$ with intersection $M = M_1 \cap M_2$. The \emph{Friedrichs angle} between the subspaces $M_1$ and $M_2$ is defined to
be the angle in $[0, \pi/2]$ whose cosine is given by
$$c(M_1,M_2) := \sup\{|\ps{x}{y}| : x \in M_1 \cap M^{\perp}\cap B_H, y \in M_2 \cap M^{\perp}\cap B_H\},$$
where $B_H := \{h \in H : \|h\| \le 1\}$ is the unit ball of $H$. The \emph{minimal angle} (or Dixmier angle) between the subspaces $M_1$ and $M_2$ is defined to
be the angle in $[0, \pi/2]$ whose cosine is given by
$$c_0(M_1,M_2) := \sup\{|\ps{x}{y}| : x \in M_1 \cap B_H, y \in M_2 \cap B_H\}.$$
\end{definition}
We note that $c(M_1,M_2) = c_0(M_1\cap M^{\perp},M_2\cap M^{\perp})$, and that $c_0(M_1,M_2) = 1$ if $M \neq \{0\}$. We also have $c(M_1,M_2) = c(M_1^{\perp},M_2^{\perp})$. We refer to the survey paper \cite{deutsch:survey_angle} for more information about different notions of angle between subspaces of infinite dimensional Hilbert spaces and their properties, and to \cite[Lecture VIII]{nikolski:book1} for different occurences of the Friedrichs angle in functional-theoretical problems.

It was proved by Aronszajn \cite{aro} (upper bound) and by Kayalar and Weinert \cite{kayalar/weinert} (equality) that
$$\|(P_{M_2}P_{M_1})^n -  P_M\| = c(M_1,M_2)^{2n-1} \quad (n \ge 1).$$
This formula shows that the sequence $(T^n)$ of iterates of $T = P_{M_2}P_{M_1}$ converges \emph{uniformly} to $T^{\infty} = P_M$ if and only if $c(M_1,M_2) < 1$, i.e., if the Friedrichs angle between $M_1$ and $M_2$ is
positive. When this happens, the iterates of $T = P_{M_2}P_{M_1}$ converge ``quickly'' (i.e. at
the rate of a geometrical progression) to $T^{\infty} = P_M$, in the following sense:
\smallskip

\begin{itemize}
 \item[(QUC)](quick uniform convergence) there exist $C > 0$ and $\alpha \in ]0,1[$ such that
$$ \|T^n - T^{\infty}\| \le C\alpha^n \quad (n \ge 1).$$
\end{itemize}
\smallskip

It is also known \cite{deutsch:survey} that $c(M_1,M_2) < 1$ if and only if $M_1+M_2$ is closed, if and only if $M_1^{\perp}+M_2^{\perp}$ is closed, if and only if $(M_1\cap M^{\perp})+(M_2\cap M^{\perp})$ is closed.

When $M_1+M_2$ is not closed, we have strong, but not uniform convergence. It was recently proved by Bauschke, Deutsch and Hundal (see \cite{bauschke/deutsch/hundal} for the history of this result) that, given any sequence of
reals decreasing to zero, there exists a point in the space with the property
that the convergence in the method of alternating projections (von Neumann's theorem) is at least as slow as
this sequence of reals. Thus the iterates of the product of two orthogonal projections converge quickly, or arbitrarily slowly. We call this alternative the (QUC)/(ASC) dichotomy : one has quick uniform convergence or arbitrarily slow convergence. We shall consider several meanings of (ASC) in this paper.

The results concerning the rate of convergence in Halperin's
theorem for $N \ge 3$ are not as complete as the results described
above for $N = 2$. We refer to \cite{deutsch:survey,DeHu,XuZi},
\cite[Chapter 9]{deutsch:book} and their references
for several results concerning the rate of convergence in the
method of cyclic alternating projections. For instance,
\cite[Example 3.7]{DeHu} shows that for $N \ge 3$ the error bound for the
method of cyclic alternating projections is not a function of the
various Friedrichs angles $c(M_i,M_j)$ between pairs of subspaces.

\subsection{What this paper is about} The main goal of the present paper is to discuss
the rate of convergence in Halperin's theorem and to generalize
some of the previous known results ($N=2$) to the case of several
subspaces ($N\ge 3$). We show by operator-theoretical methods that
the (QUC)/(ASC) dichotomy always holds as soon as the iterates of
$T$ are strongly convergent. Several interpretations of (ASC) are
proposed, and general dichotomy theorems are obtained in the
Hilbert or Banach space situation, depending on several spectral
properties imposed upon the operator $T$. This implies at once the
dichotomy (QUC)/(ASC) in all above-mentioned generalizations of
the method of alternating projections. We also give a
generalization of the Friedrichs angle to several subspaces,
$c(M_1, \cdots , M_N)$, and prove that condition (QUC) holds in
Halperin's theorem if and only if $c(M_1, \cdots , M_N) < 1$.
Estimates for the error $\|(P_{M_N} \cdots P_{M_2}P_{M_1})^n -
P_M\|$ are given in this case and several statements equivalent to
the condition $c(M_1, \cdots , M_N) < 1$ are obtained. Some of
them are expressed in terms of random products $P_{i_k}\cdots
P_{i_1}$ of projections. More specific descriptions of these
results, and information about how the paper is organized, are
given below.

\subsection{Conditions for arbitrarily slow convergence}
Several dichotomy theorems of the type quick uniform convergence \emph{versus} arbitrarily slow convergence are proved in this paper. The quick uniform condition is the condition (QUC) presented above. We shall consider in Section~2 the following conditions for (ASC):

\begin{itemize}
 \item[(ASC1)](arbitrarily slow convergence, variant 1) for every $\ep > 0$ and every sequence $(a_n)_{n\ge 1}$ of positive numbers such that $\lim_{n\to\infty} a_n = 0$, there exists a vector $x \in X$ such that $\|x\| < \sup_n a_n + \ep$ and $\|T^nx - T^{\infty}x\| \ge a_n$ for all $n$.

 \item[(ASC2)](arbitrarily slow convergence, variant 2) for every sequence $(a_n)_{n\ge 1}$ of positive numbers such that $\lim_{n\to\infty} a_n = 0$, there exists a \emph{dense} subset of points $x \in X$ such that $\|T^nx - T^{\infty}x\| \ge a_n$ for all but a \emph{finite} number of $n$'s.

 \item[(ASC3)](arbitrarily slow convergence, variant 3) for every sequence $(a_n)_{n\ge 1}$ of positive numbers such that $\lim_{n\to\infty} a_n = 0$, there exist two vectors $x \in X$ and $y \in X^*$ (the dual of $X$) such that $\re \ps{T^nx - T^{\infty}x}{y} \ge a_n$ for all $n \ge 1$. Furthermore, if there is a Banach space $Y$ such that $X$ is a (isometrical) subspace of $Y^*$, then the vector $y$ can be chosen in $Y$;
 \item[(ASCH)](arbitrarily slow convergence, Hilbertian version) for every
 $\ep > 0$ and every sequence $(a_n)_{n\ge 1}$ of positive numbers such that
 $\lim_{n\to\infty} a_n = 0$, there exists a vector $x \in H$ such that
 $\|x\| < \sup_n a_n + \ep$ and $\re \ps{T^nx - T^{\infty}x}{x} \ge a_n$ for all $n \ge 1$ .
Here $H$ is supposed to be a complex Hilbert space.
\end{itemize}

\par\smallskip
The dichotomy results of Section~2 are based upon general results
about the existence of large (weak) orbits of operators (see
\cite{muller,muller1,muller2}).
%The results concerning weak orbits
%are obtained in the Hilbertian setting in \cite{muller1}, and in
%the Banach space setting in \cite{muller2}.

Let us recall here the main result of \cite{muller2} concerning
large weak orbits in the Banach space setting:

% from \cite{muller1,muller2} concerning the existence of weak orbits of operators having
% arbitrarily slow growth. These results are obtained in the Hilbertian setting in \cite{muller1}, and in the Banach space setting in \cite{muller2}. Let us recall here
% the main result of \cite{muller2}:

\begin{theorem}[\cite{muller2}]\label{thmVlad}
 Let $X$ be a Banach space which does not contain $c_0$, and $T$ a bounded operator on
 $X$ such that $1$ belongs to the spectrum of $T$ and $\|T^nx\|$ tends to zero as
 $n$ tends to infinity for every $x\in X$. Then for any sequence
$(a_n)_{n\geq 0}$ such that $a_n$ tends to zero as $n$ tends to
infinity, there exists a vector $x\in X$ and a functional
$x^{*}\in X^{*}$ such that $\re \langle T^n x, x^{*} \rangle\geq
a_n$ for every $n\geq 0$.
\end{theorem}

We prove in Section 2 that if the iterates of $T \in \BX$ are
strongly convergent, then one has (QUC) or (ASC1). Also, if the
iterates are strongly convergent, then the dichotomy (QUC)/(ASC2)
holds. Condition (QUC) holds if and only if 
$\Ran(\lambda I - T)$, the range of $\lambda I - T$,
is closed for each $\lambda$ in the unit circle $\T$. In the case
when $T \in \BX$ is a power bounded, mean ergodic operator with
spectrum $\sigma(T)$ included in $\D \cup \{1\}$, it is proved
using the Katznelson-Tzafriri theorem \cite{KaTz} that the
iterates of $T$ are strongly convergent. Therefore the previous
dichotomies (QUC)/(ASC1) and (QUC)/(ASC2) apply. Moreover, the
dichotomy (QUC)/(ASC3) holds whenever the Banach space $X$
contains no isomorphic copy of $c_0$. If $X =H$ is a Hilbert
space, then also the dichotomy (QUC)/(ASCH) holds. We prove here
that the (QUC) condition holds if and only if $\Ran(I-T)$ is
closed. Applications to products of projections of norm one are
given. In particular, the dichotomy (QUC)/(ASC) holds, with
several variants of (ASC), for the cases covered by the theorems
of von Neumann, Halperin, Bruck-Reich and those of \cite{BaLy}.

\subsection{ A generalization of the Friedrichs angle}
In order to quantify the rate of convergence in the method of alternating projections, 
an extension of the cosine of Friedrichs angle to several subspaces $(M_1, \dots, M_N)$ will be given in Section~3. It is a parameter $c(M_1, \dots, M_N)$ which lies between $0$ and $1$, defined as follows:
\begin{eqnarray*}
c(M_1, \cdots, M_N) &=& \sup\left\{\frac{2}{N-1}\frac{\sum_{j<k}\re \ps{m_j}{m_k}}{\|m_1\|^2 + \cdots + \|m_N\|^2} : \right.\\
 & & \left. m_j \in M_j\cap M^{\perp}, \|m_1\|^2 + \cdots + \|m_N\|^2 \neq 0 \right\}\\
 &=& \sup\left\{ \frac{1}{N-1}\frac{\sum_{j\neq k} \ps{m_j}{m_k}}{\sum_{j=1}^N \ps{m_j}{m_j}} : \right.\\
 & & \left. m_j \in M_j\cap M^{\perp}, \|m_1\|^2 + \cdots + \|m_N\|^2 \neq 0 \right\} .
\end{eqnarray*}
The fact that this definition coincides with the classical one for
two subspaces will be proved in Lemma~\ref{lemma:FrAsRe}. For
pairwise orthogonal $N$ subspaces (the ``angle'' is $\pi/2$ in this
case) we have $c(M_1, \dots, M_N) = 0$, while the other extremal
case $c(M_1, \dots, M_N) = 1$ corresponds to the case of
arbitrarily slow convergence in the method of cyclic alternating
projections for $N$ subspaces (the ``angle'' is zero). Other
related quantities are considered: the \emph{configuration
constant} $\kappa(M_1, \dots, M_N)$, the \emph{inclination}
$\ell(M_1, \dots, M_N)$, and the Friedrichs angle between the
cartesian product $\bC = M_1 \times \cdots \times M_N \subset H^N$
and the ''diagonal subspace'' $\bD = \mbox{ diag }(H) = \{
(y,\dots, y) : y \in H\} \subset H^N$.

In Section~$4$ we characterize in several ways when the dichotomy (QUC)/(ASC) arises. The characterizations are in terms of geometric properties of $(M_1,\cdots,M_N)$, of spectral properties of $T$, or of random products $P_{i_k}\cdots P_{i_1}$. We give an estimate for the geometric convergence of $\|T^n-P_M\|$ to zero when $c(M_1, \dots, M_N) < 1$.

\section{General dichotomy theorems and applications}

\subsection{Dichotomy theorems}
\begin{theorem}[(QUC)/(ASC1) and (QUC)/(ASC2)]\label{thm:21}
 Let $X$ be a Banach space and let $T \in \BX$ be such that the sequence of iterates $(T^n)$ is strongly convergent to $T^{\infty}\in\BX$. Then the following dichotomy holds : either (QUC), or (ASC1).
The quick uniform convergence (condition (QUC)) holds if and only if
\beq\label{eq:closed}
 \textrm{ for every  } \ld \in \T, \Ran(\ld -T)  \textrm{ is closed. }
\eeq
In these statements, the condition (ASC1) can be replaced by (ASC2).
\end{theorem}

\begin{proof}
Suppose that
 the sequence of iterates $(T^n)_{n\ge 0}$ is strongly convergent to $T^{\infty} \in \BX$. Then $T$ is \emph{mean ergodic}, i.e., the Ces\`aro means $(I+T+ \cdots +T^{n-1})/n$ are strongly convergent. Therefore (\cite[page 73]{krengel}) the space $X$ can be decomposed as the direct sum of the kernel of $T-I$ and the closure of the range of the same operator, $X = \Ker(T-I) \oplus \overline{\Ran(T-I)}$. Moreover, $T^{\infty}$ is the projection onto $\Ker(T-I)$ along $\overline{\Ran(T-I)}$. Notice also that $T^{\infty}$ acts on the space $\Ker(T-I)$ as the identity.
%for $x \in \Ker(T-I)$ we have $Tx = T^{\infty}x = x$.
With respect to the decomposition $X = \Ker(T-I) \oplus \overline{\Ran(T-I)}$ we can write
$$
T = \left(
\begin{array}{cc}
T^{\infty} & 0 \\
0 & A \\
\end{array}
\right)
$$
for some $A \in \mathcal{B}(\overline{\Ran(T-I)})$. It is not difficult to prove that for every $\ld \in \C$, the range $\Ran(T-\ld I)$ is closed if and only if $\Ran(A-\ld I)$ is. The strong convergence of $T^n$ and the Banach-Steinhaus theorem imply that $T$ is \emph{power bounded}, that is $\sup_{n\ge 1}\|T^n\| < \infty$. Thus $\sigma(T)$, the spectrum of $T$, is included in the closed unit disk. As $\sigma(T) = \{1\} \cup \sigma(A)$, the same inclusion holds for $\sm(A)$. In particular, the spectral radius of $A$ verifies $r(A) \le 1$.

We distinguish two cases.

\smallskip

\noindent \emph{Case (1).} We have $r(A) < 1$.
Notice that we have
\beq\label{eq:Tn}
T^n - T^{\infty} = \left(
\begin{array}{cc}
0 & 0 \\
0 & A^n \\
\end{array}
\right).
\eeq
Since $r(A) < 1$, there exist $C > 0$ and $\alpha \in ]0,1[$ such that
$$ \|A^n\| \le C\alpha^n \quad (n \ge 1).$$
This estimate and \eqref{eq:Tn} gives the quick uniform convergence condition (QUC).

\smallskip

\noindent \emph{Case (2).} We have $r(A) = 1$. Recall that $\|A^ny\| \to 0$ as $n\to \infty$, for each $y \in \overline{\Ran(T-I)}$. The conditions (ASC1) and (ASC2) follow now from \cite[Thm 14, p. 333]{muller}.

\smallskip

Suppose that Case (1) is fulfilled, i.e. $r(A) < 1$. Then $A-\ld$ is invertible for every $\ld \in \T$. In particular, $\Ran(A-\ld) = \overline{\Ran(T-I)}$ is closed for each $\ld \in \T$. Thus $\Ran(T-\ld)$ is also closed, for each $\ld \in \T$.

Suppose now that all subspaces $\Ran(T-\ld)$, $\ld \in \T$, and so all $\Ran(A-\ld)$, $\ld \in \T$, are closed. Then $r(A) < 1$. Indeed, suppose that $r(A) = 1$ and let $\ld \in \T \cap \sm(A)$ be a point in the unimodular spectrum of $A$. Then the condition $\|A^ny\| \to 0$ as $n\to \infty$ for each $y$ shows that $\ld$ cannot be an eigenvalue: if $Ay = \ld y$, then $y = 0$. Indeed, we have $\|y\| = \|\ld^{-n}A^ny\| = \|A^ny\| \to 0$ as $n \to \infty$. Thus $\ld \in \sm(A) \setminus \sm_p(A)$ and $\Ran(A-\ld)$ is closed. Therefore $A-\ld$ is an upper semi-Fredholm operator. As $A-\ld I$ is a limit of invertible operators $A-\frac{n+1}{n}\ld I$, the index $\textrm{ ind}(A-\ld I)$ of $A-\ld I$ is $0.$ Hence $A-\ld I$ is invertible, a contradiction with the assumption that $\ld \in \sm(A)$. Thus $r(A) < 1$.
\end{proof}

\begin{remark}\label{rem:2}
The following is a different argument for the last part of the proof, without the use of Fredholm theory. As $\ld \in \sm(A) \setminus \sm_p(A)$ and $\Ran(A-\ld)$ is closed, the operator $A-\ld$ is lower bounded, and thus $\ld$ is not in the approximate point spectrum of $A$. As every point in the boundary of the spectrum is in the approximate point spectrum, we obtain the desired contradiction. We refer the reader to \cite{muller} as a basic reference for the spectral theory of linear operators we are using in the present paper.
\end{remark}

\begin{theorem}[(QUC)/(ASC3) and (QUC)/(ASCH)]\label{thm:22}
 Let $X$ be a Banach space and let $T \in \BX$ be a power bounded, mean ergodic operator with spectrum $\sigma(T)$ included in $\D \cup \{1\}$. Then the sequence of iterates $T^n$ is strongly convergent to a certain operator $T^{\infty} \in \BX$, and the dichotomies of Theorem \ref{thm:21} apply. Moreover, the dichotomy (QUC)/(ASC3) holds whenever $X$ contains no isomorphic copy of $c_0$. If $X =H$ is a Hilbert space, then also the dichotomy (QUC)/(ASCH) holds.

In all these statements, the quick uniform convergence condition (QUC) holds if and only if
\beq\label{eq:closed2}
 \Ran(I -T)  \textrm{ is closed. }
\eeq
\end{theorem}

\begin{proof}
Again, using the mean ergodicity and \cite[page 73]{krengel}, the space $X$ can be decomposed as the direct sum $X = \Ker(T-I) \oplus \overline{\Ran(T-I)}$. According to the Katznelson-Tzafriri theorem \cite{KaTz}, the power boundedness condition and the spectral condition $\sigma(T) \subset \D \cup \{1\}$ imply $\lim_{n\to \infty} \|T^{n+1} - T^n\| = 0$. This shows that the sequence of iterates $(T^n)$ of $T$ converges strongly to $0$ on the range of $T-I$. The same holds for the closure $\overline{\Ran(T-I)}$. As $T$ acts like identity on $\Ker(T-I)$, we get that $(T^n)_{n\ge 0}$ converges strongly to $T^{\infty}$, the projection onto $\Ker(T-I)$ along $\overline{\Ran(T-I)}$. Thus we can apply Theorem \ref{thm:21} to obtain the dichotomies (QUC)/(ASC1) and (QUC)/(ASC2).

Let us show that (QUC)/(ASC3) also holds if $X$ contains no
isomorphic copy of $c_0$. Using the notation of the proof of
Theorem \ref{thm:21}, if the condition (QUC) is not satisfied,
then $r(A) = 1$ (Case (2) in the proof of Theorem \ref{thm:21}).
As $\sm(T) \subset \D \cup \{1\}$, the same inclusion holds for
the spectrum of $A$. Therefore $1 \in \sm(A)$. Remark also that
$\|A^ny\| \to 0$ as $n \to \infty$ since $(T^n)$ converges
strongly to $T^{\infty}$. We can now apply Theorem \ref{thmVlad}
provided that $X$ contains no isomorphic copy of $c_0$. To obtain
the dichotomy (QUC)/(ASCH) if $X =H$ is a Hilbert space, we use
\cite[Theorem 2]{muller1} (see also \cite[Theorem
1]{badea/muller:topol} for the case of weak convergence).
\end{proof}

\subsection{ Applications to the method of alternating projections}
We introduce first some notation, and recall for the convenience of the reader some Banach space terminology. Let $N \ge 2$. Let $X$ be a Banach space and let $P_1, \cdots, P_N$ be $N$ fixed projections ($P_j^2 = P_j$) acting on $X$. We denote by $\pocc = \pocc(P_1,\cdots,P_N)$ the convex multiplicative semigroup generated by $P_1, \cdots, P_N$. Recall that this is the convex hull of the set of all products with factors from $P_1, \cdots, P_N$, and that the convex hull of every multiplicative semigroup of operators is a semigroup.

The space $X$ is said to be \emph{uniformly convex} if for every
$\ep \in (0,1)$ there exists $\delta \in (0,1)$ such that for any
two vectors, $x$ and $y$,  with $\|x\|\le 1$ and $\|y\|\le 1$,
$\|x+y\|/2>1-\delta$ implies $\|x-y\|<\ep$. An (equivalent)
definition of a \emph{uniformly smooth} Banach space is the
following: $X$ is uniformly smooth if its dual, $X^*$, is
uniformly convex. We refer to \cite{LT} for more information.

We call $P\in \BX$ a \emph{norm one projection} (non-zero orthoprojection) if $P^2 = P$ and $\|P\| = 1$. A self-adjoint projection in a Hilbert space is called, as usual, an \emph{orthogonal projection}. Recall that an operator $T$ on
a Banach space $X$ is called \emph{hermitian} if its numerical range is
real. This is equivalent to ask that $\|\exp(itT)\| = 1$ for every
real $t$. Hermitian operators on Hilbert spaces coincide with the self-adjoint ones; see for instance \cite{berkson} and the references therein.

\begin{theorem}\label{thm:appl_dic}
Let $N \ge 2$. Let $X$ be a complex Banach space, and let $P_1, \cdots, P_N$ be $N$ projections on $X$. Let $T$ be an operator in $\pocc(P_1,\cdots,P_N)$. If one of the following conditions below holds true, then the sequence of iterates of $T$ converges strongly and every dichotomy (QUC)/(ASC1), (QUC)/(ASC2), (QUC)/(ASC3) and (QUC)/(ASCH) (if $X=H$ is a Hilbert space) applies:
\begin{itemize}
 \item[(i)] the space $X$ is uniformly convex and each $P_j$, $1\le j \le N$, is a norm one projection;

   \item[(ii)] the space $X$ is uniformly smooth, and each $P_j$, $1\le j \le N$, is a norm one projection;

\item[(iii)] the space $X$ is reflexive and for each $j$ there
exists $r_j$ with $0 <r_j < 1$ such that $\|P_j - r_jI\| \le
1-r_j$. In particular, this holds if each $P_j$ is hermitian,
$1\le j \le N$.
\end{itemize}
\end{theorem}

\begin{proof}
It was proved in \cite{BaLy} that in all three situations the spectrum of $T \in \pocc(P_1,\cdots,P_N)$ is included in $\D \cup \{1\}$ and that the iterates of $T$ are strongly convergent. We apply the above dichotomy theorems. Notice that uniformly convex and uniformly smooth Banach spaces are reflexive, and that reflexive Banach spaces contain no copies of $c_0$.
\end{proof}

\section{A generalization of Friedrichs angle for $N$ subspaces}
As mentioned in the introduction, the rate of convergence in the method of alternating projections for two closed subspaces $M_1$ and $M_2$ is controlled by the Friedrichs angle $c(M_1,M_2)$. We introduce and study in this section a generalization of Friedrichs angle for $N$ subspaces.

\subsection{Definition}
In order to introduce our generalization of the cosine of the Friedrichs angle to several closed subspaces, we start by giving an equivalent definition of the Friedrichs angle $c(M_1,M_2)$.

\begin{lemma}\label{lemma:FrAsRe}
(a) \quad Let $M_1$ and $M_2$ be two closed subspaces of $H$. Then
 \begin{eqnarray*}
 c_0(M_1,M_2) &=& \sup \left\{ \frac{2\re \ps{m_1}{m_2}}{\|m_1\|^2+\|m_2\|^2} : m_1 \in M_1, m_2 \in M_2, (m_1,m_2) \neq (0,0)\right\}.
 \end{eqnarray*}
(b) \quad Let $M_1$ and $M_2$ be two closed subspaces in $H$. Then
 \begin{eqnarray*}
 c(M_1,M_2)&=& \sup\left\{\frac{2\re \ps{m_1}{m_2}}{\|m_1\|^2+\|m_2\|^2} : m_j \in M_j\cap M^{\perp}, (m_1,m_2) \neq (0,0)\right\} \\
&=& \sup \left\{ \frac{\ps{m_1}{m_2}+\ps{m_2}{m_1}}{\ps{m_1}{m_1}+\ps{m_2}{m_2}} : m_j \in M_j\cap M^{\perp}, (m_1,m_2) \neq (0,0)\right\}.
 \end{eqnarray*}
\end{lemma}

\begin{proof} We give the proof only for the first equality of the second part.
Denote by $s$ the first supremum from the statement of part (b).
For every admissible pair $(m_1,m_2)$ with $(m_1,m_2) \neq (0,0)$
we have

\begin{eqnarray*}
\frac{2}{\|m_1\|^2+\|m_2\|^2}\re \ps{m_1}{m_2} &\le &\frac{1}{\|m_1\|\cdot\|m_2\|}\re \ps{m_1}{m_2} \\
 &\le& \frac{|\ps{m_1}{m_2}|}{\|m_1\|\cdot\|m_2\|} \\
 &\le& c(M_1,M_2).
\end{eqnarray*}
Therefore $s \le c(M_1,M_2)$.

For the reverse inequality, let $\ep > 0$. Then there exist two elements $x_1 \in M_1\cap M^{\perp}$ and $x_2 \in M_2\cap M^{\perp}$ with $\|x_1\| = 1$ and $\|x_2\| = 1$ such that $c(M_1,M_2) < |\ps{x_1}{x_2}|+\ep$. Let $\theta \in \R$ be such that $\ps{x_1}{x_2} = |\ps{x_1}{x_2}|e^{i\theta}$, and set $m_1 = e^{-i\theta}x_1$ and $m_2 = x_2$. Then $m_1 \in M_1\cap M^{\perp}$, $m_2 \in M_2\cap M^{\perp}$ and $\|m_1\| = 1$, $\|m_2\| = 1$. We obtain
$$ s \ge \frac{2\re \ps{m_1}{m_2}}{\|m_1\|^2+\|m_2\|^2} = \re \ps{e^{-i\theta}x_1}{x_2} = |\ps{x_1}{x_2}| > c(M_1,M_2) - \ep .$$
As $\ep$ is arbitrary, we obtain $s = c(M_1,M_2)$.
\end{proof}

\begin{definition}
Let $N \ge 2$. Let $M_1, \cdots, M_N$ be $N$ closed subspaces of $H$ with intersection $M = M_1\cap \cdots \cap M_N$. The \emph{Dixmier number} associated to $(M_1, \cdots, M_N)$ is defined as
\begin{eqnarray*}
c_0(M_1, \cdots, M_N) &=& \sup\left\{ \frac{2}{N-1}\frac{\sum_{j<k}\re \ps{m_j}{m_k}}{\|m_1\|^2 + \cdots + \|m_N\|^2} : \right.\\
 & & \left. m_j \in M_j, \|m_1\|^2 + \cdots + \|m_N\|^2 \neq 0 \right\}.
\end{eqnarray*}
The \emph{Friedrichs number} $c(M_1, \cdots, M_N)$ associated to $(M_1, \cdots, M_N)$ is defined as
\begin{eqnarray*}
c(M_1, \cdots, M_N) &=& \sup\left\{\frac{2}{N-1}\frac{\sum_{j<k}\re \ps{m_j}{m_k}}{\|m_1\|^2 + \cdots + \|m_N\|^2} : \right.\\
 & & \left. m_j \in M_j\cap M^{\perp}, \|m_1\|^2 + \cdots + \|m_N\|^2 \neq 0 \right\}\\
 &=& \sup\left\{ \frac{1}{N-1}\frac{\sum_{j\neq k} \ps{m_j}{m_k}}{\sum_{j=1}^N \ps{m_j}{m_j}} : \right.\\
 & & \left. m_j \in M_j\cap M^{\perp}, \|m_1\|^2 + \cdots + \|m_N\|^2 \neq 0 \right\}.
\end{eqnarray*}
\end{definition}

\subsection{Other parameters and properties of the Friedrichs number.}
We found convenient to introduce the following parameters, called the (reduced or not) configuration constants, although they can be expressed in terms of the Dixmier and Friedrichs numbers (see Proposition \ref{prop:FRconf}, (f)).

\begin{definition}
Let $N \ge 2$. Let $M_1, \cdots, M_N$ be $N$ closed subspaces of $H$ with intersection $M = M_1\cap \cdots \cap M_N$. The number
$$ \kappa_0(M_1, \cdots, M_N) = \sup\left\{ \frac{1}{N} \frac{\|\sum_{j=1}^N m_j\|^2}{\sum_{j=1}^N\|m_j\|^2} : m_j \in M_j, \|m_1\|^2 + \cdots + \|m_N\|^2 \neq 0 \right\} $$
is called the \emph{non-reduced configuration constant} of $(M_1, \cdots, M_N)$. The number
$$ \kappa(M_1, \cdots, M_N) = \sup\left\{ \frac{1}{N} \frac{\|\sum_{j=1}^N m_j\|^2}{\sum_{j=1}^N\|m_j\|^2} : m_j \in M_j\cap M^{\perp}, \|m_1\|^2 + \cdots + \|m_N\|^2 \neq 0 \right\} $$
is called the \emph{configuration constant} of $(M_1, \cdots, M_N)$.
\end{definition}

The configuration constant is related to the maximal possible norms of Gramian matrices. Recall that the \emph{Gramian matrix} of an $N$-tuple of vectors $(v_1, \cdots, v_N)$ is the $N\times N$ matrix
$$G(v_1, \cdots, v_N) = [\ps{v_i}{v_j}]_{1\le i,j\le N} .$$

\begin{proposition}\label{prop:Gram}
Let $N \ge 2$. Let $M_1, \cdots , M_N$ be $N$ closed subspaces of $H$ with intersection $M = M_1 \cap \cdots \cap M_N$. Then
$$ \kappa(M_1, \cdots, M_N) = \sup\left\{ \frac{1}{N} \left\| G(v_1, \cdots, v_N) \right\|: v_j \in M_j\cap M^{\perp}, \|v_j\| = 1, j= 1, \cdots, N \right\} .$$
\end{proposition}

\begin{proof}
Let $m_j \in M_j\cap M^{\perp}$, $j= 1, \cdots, N$, with $\|m_1\|^2 + \cdots + \|m_N\|^2 \neq 0$. Set $v_j = \frac{m_j}{\|m_j\|}$ if $\|m_j\| \neq 0$, or $v_j = 0$ if $m_j = 0$. Denote ${\mathbf x} = (\|m_1\|, \cdots , \|m_N\|) \in \C^N\setminus \{0\}$. We have $\ps{{\mathbf x}}{{\mathbf x}} = {\mathbf x}^t{\mathbf x} \neq 0$ and
 \begin{eqnarray*}
    \frac{1}{N} \frac{\|\sum_{j=1}^N m_j\|^2}{\sum_{j=1}^N\|m_j\|^2} &=& \frac{1}{N} \frac{\ps{\sum_i^N \|m_i\|v_i}{\sum_j^N \|m_j\|v_j}}{\sum_{j=1}^N\|m_j\|^2} \\
    &=& \frac{1}{N} \frac{{\mathbf x}^tG(v_1, \cdots, v_N){\mathbf x}}{{\mathbf x}^t{\mathbf x}} .
 \end{eqnarray*}
The conclusion follows by taking the supremum and noting that the Gramian matrix $G(v_1, \cdots, v_N)$ is a Hermitian matrix.
\end{proof}

Consider the product Hilbert space $H^N$ which is the Hilbertian direct sum of $N$ copies of $H$, with scalar product
$$ \ps{\mathbf x}{\mathbf y} = \ps{(x_1,\cdots, x_N)}{\mathbf (y_1,\cdots, y_N)} := \sum_{j=1}^N \ps{x_j}{y_j} .$$
We denote by $\bC$ the Cartesian product $\bC = M_1 \times \cdots \times M_N \subset H^N$, and by $\bD$ the diagonal subset $\bD = \mbox{ diag }(H) = \{ (y,\dots, y) : y \in H\} \subset H^N$. Recall that $M = M_1 \cap \cdots \cap M_N$.

\begin{lemma}\label{lemma:CandD}
The projections onto $\bC$, $\bD$ and $\bC \cap \bD$ are given by
$$ P_{\bC}(x_1,\cdots, x_N) = (P_1x_1,\cdots, P_Nx_N) ,$$
$$ P_{\bD}(x_1,\cdots, x_N) = ((x_1 + \cdots + x_N)/N, \cdots , (x_1 + \cdots + x_N)/N) ,$$
and, respectively, by
$$ P_{\bC \cap \bD}(x_1,\cdots, x_N) = ((P_Mx_1 + \cdots + P_Mx_N)/N, \cdots , (P_Mx_1 + \cdots + P_Mx_N)/N) ,$$
where $(x_1,\cdots, x_N) \in H^N$.
\end{lemma}

\begin{proof}
The formulae for $P_{\bC}$ and $P_{\bD}$ were proved in
\cite{pierra}. For the third one, we note first that
$$\bC \cap \bD = \mbox{ diag }(M) = \{(m,\cdots ,m) : m \in M\} .$$
Therefore
$$ \|{\mathbf x} - P_{\bC \cap \bD}{\mathbf x}\|^2 = \dist({\mathbf x},\mbox{ diag }(M))^2 = \inf \{ \sum_{j=1}^N \|x_j - m\|^2 : m \in M\} .$$
We obtain
$$ \|{\mathbf x} - P_{\bC \cap \bD}{\mathbf x}\|^2 = \sum_{j=1}^N \|x_j - P_Mx_j\|^2 + \inf \left\{ \sum_{j=1}^N \|P_Mx_j - m\|^2 : m \in M\right\} .$$
The infimum is realized when the gradient is zero, $\sum_{j=1}^N (m-P_Mx_j) = 0$, that is when $m = N^{-1}\sum_{j=1}^N P_Mx_j$.
\end{proof}

\begin{proposition}\label{prop:FRconf}
Let $N \ge 2$. Let $M_1, \cdots, M_N$ be $N$ closed subspaces of $H$ with intersection $M = M_1\cap \cdots \cap M_N$. Then
\begin{itemize}
  \item[(a)] $c_0(M_1, \cdots, M_N) = 1$ if $M \neq \{0\}$, while $c_0(M_1, \cdots, M_N) = 0$ if and only if the subspaces $(M_1, \cdots, M_N)$ are pairwise orthogonal;
  \item[(b)] $c(M_1, \cdots, M_N) = c(M_1\cap M^{\perp}, \cdots, M_N\cap M^{\perp})= c_0(M_1\cap M^{\perp}, \cdots, M_N\cap M^{\perp})$, and thus $c(M_1, \cdots, M_N) = c_0(M_1, \cdots, M_N)$ if $M = \{0\}$;
  \item[(c)] $0 \le c_0(M_1, \cdots, M_N) \le 1$ and $0 \le c(M_1, \cdots, M_N) \le 1$;
  \item[(d)] $\frac{1}{N} \le \kappa_0(M_1, \cdots, M_N) \le 1$ and $\frac{1}{N} \le \kappa(M_1, \cdots, M_N) \le 1$;
  \item[(e)] $\kappa_0(M_1, \cdots, M_N) = c_0(\bC,\bD)^2$ and $\kappa(M_1, \cdots, M_N) = c(\bC,\bD)^2$;
  \item[(f)] $c(M_1, \cdots, M_N) = \frac{N}{N-1}\kappa(M_1, \cdots, M_N) - \frac{1}{N-1} = \frac{N}{N-1}c(\bC,\bD)^2 - \frac{1}{N-1}$ and similar statements hold for $c_0(M_1, \cdots, M_N)$.
\end{itemize}
\end{proposition}

\begin{proof}
We start by giving the proof of part (e). We have
\begin{eqnarray*} c(\bC,\bD)^2 &=& \sup \left\{ \frac{|\ps{(m_1,\dots, m_N)}{(y,\dots,y)}_{H^N}|^2}{N(\|m_1\|^2 + \cdots + \|m_N\|^2)\|y\|^2} : \right.\\
& & \left. y \in H, y\neq 0, m_j \in M_j\cap M^{\perp}, \|m_1\|^2 + \cdots + \|m_N\|^2 \neq 0\right\}\\
&=& \sup \left\{ \frac{\ps{\sum_{j=1}^Nm_j}{y}^2}{N(\|m_1\|^2 + \cdots + \|m_N\|^2)\|y\|^2} : \right.\\
& &  \left. y \in H, y\neq 0, m_j \in M_j\cap M^{\perp}, \|m_1\|^2 + \cdots + \|m_N\|^2 \neq 0\right\} \\
&=& \sup \left\{ \frac{\|\sum_{j=1}^Nm_j\|^2}{N(\|m_1\|^2 + \cdots + \|m_N\|^2)}: \right.\\
& & \left. y \in H, y\neq 0, m_j \in M_j\cap M^{\perp}, \|m_1\|^2 + \cdots + \|m_N\|^2 \neq 0\right\}\\
&=& \kappa(M_1,\dots, M_N).
\end{eqnarray*}
The proof of the equality $\kappa_0(M_1, \cdots, M_N) = c_0(\bC,\bD)^2$ is similar.

We prove now that $c(M_1, \cdots, M_N) = \frac{N}{N-1}\kappa(M_1, \cdots, M_N) - \frac{1}{N-1}$. Indeed, we have
\begin{eqnarray*}
 c(M_1, \cdots, M_N) &=& \sup\left\{\frac{2}{N-1}\frac{\sum_{j<k}\re \ps{m_j}{m_k}}{\|m_1\|^2 + \cdots + \|m_N\|^2} : \right.\\
 & & \left. m_j \in M_j\cap M^{\perp}, \|m_1\|^2 + \cdots + \|m_N\|^2 \neq 0 \right\} \\
  &=& \sup \left\{\frac{1}{N-1} \frac{\|\sum_{j}m_j\|^2 - \sum_j \|m_j\|^2}{\|m_1\|^2 + \cdots + \|m_N\|^2} : \right.\\
  & & \left. m_j \in M_j\cap M^{\perp}, \|m_1\|^2 + \cdots + \|m_N\|^2 \neq 0 \right\} \\
 &=& \frac{N}{N-1}\kappa(M_1, \cdots, M_N) - \frac{1}{N-1}.
\end{eqnarray*}
The proof of the equality for $c_0(M_1, \cdots, M_N)$ is similar.

The upper bound $\kappa_0(M_1, \cdots, M_N) \le 1$ in (d) follows from the Cauchy-Schwarz inequality:
\begin{eqnarray*}
  \|m_1+\dots +m_N\|^2 &\le& (\|m_1\| + \dots + \|m_N\|)^2 \\
   &\le& N(\|m_1\|^2 + \dots + \|m_N\|^2) .
\end{eqnarray*}
For the lower bound $\kappa_0(M_1, \cdots, M_N) \ge 1/N$, notice that we have, for $m_1 \in M_1\setminus \{0\}$,
$$ \kappa_0(M_1, \cdots, M_N) \ge \frac{1}{N} \frac{\|\|m_1\|^2}{\|m_1\|^2} = \frac{1}{N}.$$
The inequalities for $\kappa(M_1, \cdots, M_N)$ follow from
$$\kappa(M_1, \cdots, M_N) = \kappa_0(M_1\cap M^{\perp}, \cdots, M_N\cap M^{\perp}) .$$

Now (c) is a consequence of (f) and (d), while (b) and (a) are easy to prove. For the first equality in (b) notice that $\cap_{j=1}^N (M_j\cap M^{\perp}) = \{0\}$.
\end{proof}

\begin{proposition}\label{prop:GammaProj}
Let $N \ge 2$. Let $M_1, \cdots , M_N$ be $N$ closed subspaces of $H$ with intersection $M = M_1 \cap \cdots \cap M_N$. Then
$$ \kappa(M_1, \cdots, M_N) = \left\|\frac{P_1 + \cdots + P_N}{N} - P_M\right\|$$
and
$$ c(M_1, \cdots, M_N) = \frac{N}{N-1}\left\|\frac{P_1 + \cdots + P_N}{N} - P_M\right\| - \frac{1}{N-1}.$$
\end{proposition}

\begin{proof}
We have (see for instance \cite{kayalar/weinert})
$$ c(\bC,\bD) = \left\| P_{\bD}P_{\bC} - P_{\bC\cap\bD}\right\| .$$
Using Lemma \ref{lemma:CandD}, $P_{\bD}P_{\bC} - P_{\bC\cap\bD}$ can be written as
$$ \frac{1}{N}\left[
     \begin{array}{cccc}
       1 & 1 &\cdots & 1 \\
       1 & 1 & \cdots & 1\\
       \vdots & \vdots &  & \vdots\\
       1 & 1 & \cdots & 1 \\
     \end{array}
   \right] \left[
     \begin{array}{cccc}
       P_1 & 0 & \cdots & 0 \\
       & \ddots &   &\\
       &  & \ddots  &\\
       0 & 0 & \cdots & P_N
     \end{array}
   \right] - \frac{1}{N}\left[
     \begin{array}{cccc}
       P_M & P_M &\cdots & P_M \\
       P_M & \cdots & \cdots & P_M\\
       \vdots &  &  &  \vdots \\
       P_M & P_M &\cdots & P_M
     \end{array}
   \right] .
$$
Therefore
$$P_{\bD}P_{\bC} - P_{\bC\cap\bD} = \frac{1}{N}\left[
     \begin{array}{cccc}
       P_1-P_M & P_2-P_M &\cdots & P_N-P_M \\
       P_1-P_M & P_2-P_M &\cdots & P_N-P_M \\
       \vdots &  & & \vdots \\
      P_1-P_M & P_2-P_M &\cdots & P_N-P_M
     \end{array}
   \right],$$
and so
$$
\kappa := \kappa(M_1, \dots , M_N) = c(\bC,\bD)^2 = \frac{1}{N^2}\left\|\left[
     \begin{array}{cccc}
       P_1-P_M & P_2-P_M &\cdots & P_N-P_M \\
       P_1-P_M & P_2-P_M &\cdots & P_N-P_M \\
       \vdots &  & & \vdots \\
      P_1-P_M & P_2-P_M &\cdots & P_N-P_M
     \end{array}
   \right]\right\|^2 .$$
We obtain that $N^2\kappa$ is equal to
$$\left\|\left[\begin{array}{cccc}
       P_1-P_M & P_2-P_M &\cdots & P_N-P_M \\
       P_1-P_M & P_2-P_M &\cdots & P_N-P_M \\
       \vdots &  & & \vdots \\
      P_1-P_M & P_2-P_M &\cdots & P_N-P_M
     \end{array}
   \right]\left[\begin{array}{cccc}
       P_1-P_M & P_1-P_M &\cdots & P_1-P_M \\
       P_2-P_M & P_2-P_M &\cdots & P_2-P_M \\
       \vdots &  & & \vdots \\
      P_N-P_M & P_N-P_M &\cdots & P_N-P_M
     \end{array}
   \right]\right\| .$$
 Therefore
$$N^2\kappa = \left\|\left[\begin{array}{cccc}
       \Sigma & \Sigma &\cdots & \Sigma \\
       \Sigma & \Sigma &\cdots & \Sigma \\
       \vdots &  & & \vdots \\
      \Sigma & \Sigma &\cdots & \Sigma
     \end{array}
   \right]\right\|, \quad \mbox{ where } \Sigma := \sum_{j=1}^N (P_j-P_M)^2 .$$
Since
$$(P_j-P_M)^2 = P_j -P_jP_M -P_MP_j + P_M = (I-P_M)P_j, $$
we have
$$ \Sigma = (I-P_M)\sum_{j=1}^N P_j .$$
Let $K$ be the matrix having all entries equal to $\Sigma$. One way to compute the norm of $K$ is to note that, like every circulant matrix, $K$ is unitarily equivalent to a diagonal matrix. Indeed, denote by $F$ the $N\times N$ unitary matrix representing the discrete Fourier transform $F = N^{-1/2}[(\omega^{jk})]_{0\le j,k \le N-1}$, where $\omega = \exp(-2i\pi/N)$ is a primitive $N$th root of unity. Then
$$ F^{\ast}KF = \left[\begin{array}{cccc}
       N\Sigma & 0 &\cdots & 0 \\
       0 & 0 &\cdots & 0 \\
       \vdots &  & & \vdots \\
      0 & 0 &\cdots & 0
     \end{array} \right] . $$
Therefore
$$ \kappa = \frac{1}{N}\|\Sigma\| = \|(I-P_M)\sum_{j=1}^N P_j\| = \|\sum_{j=1}^N P_j(I-P_M)\|.$$
Since $P_jP_M = P_M$, this can be written as
$$ \kappa = \left\|\frac{\sum_{j=1}^N P_j}{N} - P_M\right\|.$$
The proof is complete.
\end{proof}

The following definition is related to the minimum gap between two subspaces (see \cite[p. 219 and Lemma 4.4]{kato}). See also the regularity (or boundedly linearly regularity) condition from \cite{bauschke/borwein}, and the references therein.
\begin{definition}
Let $N \ge 2$. Let $M_1, \cdots, M_N$ be $N$ closed subspaces of $H$ with intersection $M = M_1\cap \cdots \cap M_N$. The number
$$ \ell(M_1, \dots , M_N) = \inf_{x\notin M} \frac{\max_{1\le j\le N}\dist (x, M_j)}{\dist(x,M)}$$
is called the \emph{inclination} of $(M_1, \cdots, M_N)$.
\end{definition}

\begin{proposition}\label{prop:CandL}
Let $N \ge 2$. Let $M_1, \cdots, M_N$ be $N$ closed subspaces of $H$ with intersection $M = M_1\cap \cdots \cap M_N$. Then
$$ 1-\ell(M_1, \dots , M_N) \le c(\bC,\bD) =  \kappa(M_1, \dots , M_N)^{1/2}\le 1 - \frac{\ell(M_1, \dots , M_N)^2}{2N}$$
and
$$ 1 - \frac{2N}{N-1}\ell(M_1, \dots , M_N) \le c(M_1, \dots , M_N) \le 1 - \frac{\ell(M_1, \dots , M_N)^2}{N-1}.$$
In particular, $\ell(M_1, \dots , M_N) = 0$ if and only if $c(M_1, \dots , M_N) = 1$, if and only if $\kappa(M_1, \dots , M_N) = 1$.
\end{proposition}
\begin{proof}
Denote $\ell = \ell(M_1, \dots , M_N)$. Let $\ep > 0$. There exists $x \in H$ with $\|x-P_Mx\| = \dist(x,M) = 1$ such that $\dist(x,M_j) < \ell+\ep$ for each $j$. Set $u_j = P_j(x-P_Mx) = P_jx-P_Mx$, where $P_j$ is the orthogonal projection onto $M_j$ ($1\le j \le N)$. Then $u_j \in M_j$ and $\|u_j\| \le \|x-P_Mx\| = 1$. Recall that $\bC$ is the $\ell^2$-direct sum $\bC = M_1 \times \cdots \times M_N \subset H^N$, while $\bD$ is the diagonal $\bD = \mbox{ diag }(H) = \{ (y,\dots, y) : y \in H\} \subset H^N$. We have $\bC \cap \bD = \mbox{ diag }(M)$, and so ${\mathbf y} = (y_1, \cdots, y_N) \in \mbox{ diag }(M)^{\perp}$ if and only if $\ps{y_1+\cdots+y_N}{m} = 0$ for every $m\in M$. Thus $ (y_1, \cdots, y_N) \in \mbox{ diag }(M)^{\perp}$ if and only if $y_1+\cdots+y_N \in M^{\perp}$.

Consider
$$\mathbf{d} = (\frac{1}{\sqrt{N}} (I-P_M)x, \dots , \frac{1}{\sqrt{N}} (I-P_M)x) \in \bD \cap \mbox{ diag }(M)^{\perp}$$
 and
 $$\mathbf{c} = (\frac{1}{\sqrt{N}} u_1, \dots , \frac{1}{\sqrt{N}} u_N) \in \bC .$$
For each $m \in M$ we have
\begin{eqnarray*}
  \ps{\sum_{j=1}^N u_j}{m} &=&  \frac{1}{\sqrt{N}}\left[ \sum_{j=1}^N\ps{P_jx}{m} - N\ps{P_Mx}{m}\right]\\
   &=& \frac{1}{\sqrt{N}}\left[ \sum_{j=1}^N\ps{x}{P_jm} - N\ps{x}{m}\right] \\
   &=& 0 .
\end{eqnarray*}
Therefore $\mathbf{c} \in \bC \cap \mbox{ diag }(M)^{\perp}$.
We also have $\|\mathbf{d}\| = 1$ and $\|\mathbf{c}\|^2 = \frac{1}{N}(\|u_1\|^2 + \cdots + \|u_N\|^2)\le  1$. Thus
$$c(\bC,\bD) \ge \left|\ps{\mathbf{c}}{\mathbf{d}}\right| = \frac{1}{N} \left|\ps{\sum_{j=1}^Nu_j}{x-P_Mx}\right| \ge  \frac{1}{N} \re\ps{\sum_{j=1}^Nu_j}{x-P_Mx}.$$
For a fixed $j$ we have $\|x-P_Mx -u_j\| = \|x-P_jx\| = \dist(x,M_j) < \ell+\ep$. Therefore
$$2\re\ps{x-P_Mx}{u_j} = \|x-P_Mx\|^2 + \|u_j\|^2 -  \|x-P_Mx-u_j\|^2 > 1+\|u_j\|^2 - (\ell+\ep)^2.$$
We also have $\|u_j\| \ge \|x-P_Mx\| - \|x-P_Mx-u_j\| > 1 - (\ell+\ep)$. We obtain
\begin{eqnarray*}
  c(\bC,\bD) &\ge& \frac{1}{N} \re\ps{\sum_{j=1}^Nu_j}{x-P_Mx} \\
   &\ge& \frac{1}{2N} \sum_{j=1}^N \left( 1 + \|u_j\|^2 - (\ell+\ep)^2\right) \\
   &\ge& \frac{1}{2N} \left(N+N(1 - (\ell+\ep))^2 -N(\ell+\ep)^2\right) \\
   &=& \frac{1}{2} \left(1+(1 - (\ell+\ep))^2 -(\ell+\ep)^2\right) .
\end{eqnarray*}
As this inequality is true for every $\ep > 0$ we get
$$c(\bC,\bD) \ge \frac{1-\ell^2+(1-\ell)^2}{2} = 1-\ell .$$

Denote $c = c(\bC,\bD)$. Let $\ep > 0$. There exist $\mathbf{c} \in \bC\cap \mbox{ diag }(M)^{\perp}$ and $\mathbf{d} \in \bD \cap \mbox{ diag }(M)^{\perp}$ with $\|\mathbf{c}\| = 1$, $\|\mathbf{d}\| = 1$ such that $c < |\ps{\mathbf{c}}{\mathbf{d}}| + \ep$. Let $\theta \in \R$ be such that $\ps{\mathbf{c}}{\mathbf{d}} = e^{i\theta}|\ps{\mathbf{c}}{\mathbf{d}}|$. Then
$$\|\mathbf{c} - e^{i\theta}\mathbf{d}\|^2 = 2 - 2\re ( e^{-i\theta}\ps{\mathbf{c}}{\mathbf{d}}) = 2-2|\ps{\mathbf{c}}{\mathbf{d}}| \le 2-2(c-\ep).$$
Set $\mathbf{c} = (m_1, \dots , m_N) \in \bC \cap \mbox{ diag }(M)^{\perp}$ and $e^{i\theta}\mathbf{d} = (y, \dots , y) \in \bD \cap \mbox{ diag }(M)^{\perp}$. Then $y\in M^{\perp}$ and
$$\|m_1\|^2 +  \cdots + \|m_N\|^2 = 1, \quad \|y\| = 1/\sqrt{N}, \quad  \mbox{and} \quad \sum_{j=1}^N \|m_j - y\|^2 \le 2-2(c-\ep).$$
Let $x = \sqrt{N}y$. Then $x \in M^{\perp}$, $\dist(x, M) = \|x\| = 1$ and we have
$$\dist(x,M_j)^2 \le \|x- \sqrt{N}m_j\|^2 = N\|m_j-y\|^2 \le N(2-2(c-\ep)).$$
We finally obtain $\ell^2 \le 2N(1-c)$, and so
$$1-\ell \le c(\bC,\bD) \le 1 - \frac{\ell^2}{2N}.$$

Using the equalities
$$c(\bC,\bD) =  \kappa(M_1, \dots , M_N)^{1/2} \quad \mbox{ and }\kappa(M_1, \dots , M_N) = \frac{N-1}{N}c(M_1, \cdots, M_N) + \frac{1}{N}$$
we obtain
$$ \frac{N(1-\ell)^2-1}{N-1} \le c(M_1, \cdots, M_N) \le \frac{N(1-(\ell^2/2N)^2)-1}{N-1}.$$
Since $2-\ell \le 2$, we can write
$$ \frac{N(1-\ell)^2-1}{N-1} = 1 - \frac{N}{N-1}\ell(2-\ell) \ge 1 - \frac{2N}{N-1}\ell .$$
We also have
$$ \frac{N(1-(\ell^2/2N)^2)-1}{N-1} = 1 - \frac{\ell^2}{N-1}(1 - \frac{\ell^2}{4N}) \le 1 - \frac{\ell^2}{N-1}.$$
\end{proof}

\section{Characterising (ASC) for products of projections}
\subsection{A qualitative result}
When $T$ is the product of $N$ orthogonal projections, we know from Theorem \ref{thm:appl_dic} that the dichotomy (QUC)/(ASC) holds, and that we have quick uniform convergence if and only if the range of $T-I$ is closed. The following qualitative result gives a characterization of the (ASC) condition in terms of several parameters associated to $(M_1, \cdots , M_N)$, or spectral properties of $T$, or random products. We denote by $\|\cdot\|_e$ the \emph{essential norm} and by $\sigma_e$ the \emph{essential spectrum}.
\begin{theorem}\label{thm:qualASC}
Let $N \ge 2$. Let $M_1, \cdots, M_N$ be $N$ closed subspaces of $H$ with intersection $M = M_1\cap \cdots \cap M_N$. Denote $P_j$ the orthogonal projection onto $M_j$, $1\le j \le N$, and by $P_M$ the orthogonal projection onto $M$. Let $T = P_NP_{N-1}\cdots P_1$. The following assertions are equivalent:
\begin{itemize}
  \item[(1)]  $\Ran(T-I)$ is not closed;
  \item[(1$^\prime$)] for every $k\ge N$ and every sequence of indices $(i_k)_{k\ge 1}$  such that $\{i_1, \ldots, i_k\} = \{1,2,\ldots, N\}$, $\Ran(P_{i_k}\cdots P_{i_1}-I)$ is not closed;
  \item[(2)] one of the conditions (ASC1), (ASC2), (ASC3), (ASCH) holds for $T$;
   \item[(2$^\prime$)] (ASCH) for random products: for every $\ep > 0$, every sequence $(a_n)_{n\ge 0}$ of positive reals with $\lim_{n\to\infty}a_n = 0$, and every sequence of indices $(i_k)_{k\ge 1}$ in $\{1,2,\ldots, N\}$, there exists $x\in H$ with $\|x\| < \sup_n a_n + \ep$ such that $$\re\ps{P_{i_n}P_{i_{n-1}}\cdots P_{i_1}x-P_Mx}{x}>a_n$$ for each $n\ge 1$;
  \item[(3)] $c(M_1, \cdots , M_N) = 1$. Equivalently, $\kappa(M_1, \cdots , M_N) = 1$, or $\ell(M_1, \cdots , M_N) = 0$;
  \item[(4)] for every $\ep > 0$, every closed subspace $K \subset M^{\perp}$ of finite codimension (in $M^{\perp}$), there exists $x \in K$ such that $\|x\| = 1$ and $\max\{\dist(x,M_j) : j=1, \cdots, N\} < \ep$;
  \item[(5)] $1 \in \sigma(T-P_M)$;
  \item[(5$^\prime$)] for every $k$ and every $i_1, \cdots, i_k \in \{1,2, \cdots,N\}$ we have $1 \in \sigma(P_{i_k}\cdots P_{i_1}-P_M)$;
  \item[(6)] $\|T - P_M\| = 1$ ;
    \item[(6$^\prime$)] for every $k$ and every sequence of indices $(i_k)_{k\ge 1}$, $1\le i_k\le N$, with $\{i_1, \ldots, i_k\} = \{1,2,\ldots, N\}$ we have $\|P_{i_k}\cdots P_{i_1} - P_M\| = 1$ ;
   \item[(7)] $\|T-P_M\|_e = 1$;
  \item[(7$^\prime$)] for every $k$ and every $i_1, \cdots, i_k \in \{1,2, \cdots,N\}$ we have $\|P_{i_k}\cdots P_{i_1}-P_M\|_e = 1$;
  \item[(8)] $1 \in \sigma_e(T-P_M)$;
  \item[(8$^\prime$)] for every $k$, every $i_1, \cdots, i_k \in \{1,2, \cdots,N\}$
  %such that $\{i_1,i_2, \cdots,i_N\} = \{1,2, \cdots,N\}$,
  we have $1 \in \sigma_e(P_{i_k}\cdots P_{i_1}-P_M)$;
  \item[(9)] for every $\ep > 0$, every closed subspace $K \subset M^{\perp}$ of finite codimension (in $M^{\perp}$), there exists $x \in K$ such that $\|Tx-x\| \le \ep$;
  \item[(9$^\prime$)] for every $\ep > 0$, every closed subspace $K \subset M^{\perp}$ of finite codimension (in $M^{\perp}$), there exists $x \in K$ such that $\|P_{i_k}\cdots P_{i_1}x-x\| \le \ep$ for every $k$, every $i_1, \cdots, i_k \in \{1,2, \cdots,N\}$ ;
  \item[(10)] the sum of $\mbox{ diag}(M_1)\subset H^{N-1}$ and $M_2\oplus \cdots \oplus M_N \subset H^{N-1}$ is not closed in $H^{N-1}$ (and equivalent statements for $\mbox{ diag}(M_j)\subset H^{N-1}$, $2\le j\le N$);
  \item[(11)] $M_1^{\perp} + \cdots + M_N^{\perp}$ is not closed in $H$.
\end{itemize}
\end{theorem}

The conditions $(1), (2), (5), (6), (7), (8)$ and $(9)$, most of them of spectral nature, are conditions about $T=P_N\cdots P_1$, while the corresponding conditions denoted with primes are analog conditions about random products $P_{i_N}\cdots P_{i_1}$. The conditions $(3), (4), (10)$ and $(11)$ are about the geometry of subspaces $M_j$.
\par\smallskip
Notice that we have the dichotomy (QUC)/(ASC) in
all possible senses, and that (QUC) holds if and only if $c(M_{1},
\ldots, M_{N})<1$. A quantitative estimate reflecting the
geometric convergence of $\|T^n-P_M\|$ to zero, in terms of the
Friedrichs number, will be given after the proof of the theorem.

\begin{proof}[Proof of Theorem \ref{thm:qualASC}]
''$(1) \Leftrightarrow (2)$`` \quad The equivalence of (1) and (2) follows from Theorem \ref{thm:appl_dic}.

''$(1) \Leftrightarrow (5)$`` \quad The equivalence of (1) and (5) follows from the proof of Theorem \ref{thm:21} (see also Remark \ref{rem:2}). Notice that, with respect to the decomposition $H = M\oplus M^{\perp}$, we have $T = P_M\oplus A$, where $A = T\mid_{M^{\perp}} = T(I-P_M) = T-P_M$.

''$(1) \Rightarrow (3)$`` \quad We prove this implication in a quantitative form. Denote $$\gamma = \gamma(I-T) =\inf\left\{ \|x-Tx\| : x\in H, \dist(x,\Ker(T-I) = 1\right\}$$
the \emph{reduced minimum modulus} of $T-I$. Then $\Ran(T-I)$ is closed if and only if $\gamma > 0$. Clearly $Ty = y$ for $y \in M$. If $Tx =x$, then
$$ \|x\| = \|P_N\cdots P_1x\| \le \|P_{N-1}\cdots P_1x\| \le \cdots \le \|P_1x\| \le \|x\|.$$
We successively obtain $P_1x = x$, $P_2x = x$, \dots , $P_Nx = x$, and finally $x \in M$. Thus $\Ker(T-I) = M$.

Let $\ep > 0$. There exists $x \in H$ with $\|x-P_Mx\| = \dist(x, M) = 1$ such that $\|x-Tx\| \le \gamma+\ep$. We obtain

\begin{eqnarray*}
1 = \|x-P_Mx\| &\ge & \|P_1(x-P_Mx)\| = \|P_1x-P_Mx\| \ge \|P_2P_1x-P_Mx\| \\
&\ge &\dots \ge \|P_N \cdots P_1x -P_Mx\| = \|Tx-P_Mx\| \\
&\ge & \|x-P_Mx\| - \|x-Tx\| \ge 1 - \gamma - \ep .
\end{eqnarray*}
We also have
\begin{eqnarray*}
\|(I-P_1)(x-P_Mx)\|^2 &=& \|x-P_Mx\|^2 - \|P_1x-P_Mx\|^2 \\
&\le& 1-(1-\gamma-\ep)^2 = -(\gamma+\ep)^2+2(\gamma+\ep) \le 2\gamma+2\ep .
\end{eqnarray*}
Thus $\dist(x,M_1) = \|x-P_1x\| = \|(I-P_1)(x-P_Mx)\|\le (2\gamma+2\ep)^{1/2}$.

Let $y = x - P_Mx$; then $\|y\| = 1$. For a fixed $s$ between $1$ and $N$ we can write
 \begin{eqnarray*}
 \|P_s \cdots P_1y - P_{s+1} \cdots P_1y\|^2 &=& \|P_{s} \cdots P_1y\|^2 - \|P_{s+1} \cdots P_1y\|^2\\
  &\le& \|y\|^2 - \|P_{s+1} \cdots P_1x - P_Mx\|^2 \\
  &\le& 1 - (1-\gamma-\ep)^2 \le 2\gamma+2\ep .
 \end{eqnarray*}
Thus
\begin{eqnarray*}
  \dist(x,M_j) &=& \dist(y,M_j) \le  \|y - P_{j} \cdots P_1y\| \\
   &\le& \|y - P_1y\| + \|P_1y - P_{2}P_1y\| + \dots + \|P_{j-1} \cdots P_1y - P_{j} \cdots P_1y\|\\
   &\le& j\sqrt{(2\gamma+2\ep)},
\end{eqnarray*}
for every $j$. Hence $\max_{1\le j \le N} \dist(x,M_j) \le N\sqrt{(2\gamma+2\ep)}$ and, as $\ep$ is arbitrary,
$$ \ell := \ell(M_1, \dots , M_N) \le N\sqrt{2\gamma} .$$
We obtain
$ \frac{1}{2N^2}\ell(M_1, \dots , M_N)^2 \le \gamma(T-I) .$ Therefore $\Ran(I-T)$ not closed ($\gamma = 0$) implies $\ell(M_1, \dots , M_N) = 0$.

''(3) $\Rightarrow$ (1$^\prime$)`` \quad Let $k \ge N$ and let $(i_k)_{k\ge 1}$ be a sequence of indices with $\{i_1, \ldots, i_k\} = \{1,2,\ldots, N\}$. This implies that $\Ker(I - P_{i_k}P_{i_{k-1}}\cdots P_{i_1}) = M$. Let $ \ell = \ell(M_1, \dots , M_N)$ and
let $\ep > 0$. There exists $x\in H$ with $\|x-P_Mx\| = \dist(x,M) = 1$ such that $\max_j \dist (x, M_j) < \ell+\ep$. We have
$$\|x-P_{i_1}x\| = \dist(x,M_{i_1}) < \ell+\ep$$
and
$$ \|P_{i_2}P_{i_1}x-P_{i_1}x\| = \dist(P_{i_1}x,M_{i_2}) \le \|x-P_{i_1}x\| + \dist(x,M_{i_2}) < 2(\ell+\ep) .$$
Set $x_0 = x$ and $x_s = P_{i_s}P_{i_{s-1}} \cdots P_{i_1}x$ for $s\ge 1$. Suppose that
\beq \label{eq:41}
 \|x_s - x_{s-1}\| \le 2^{s-1}(\ell+\ep)
\eeq
holds for $1\le s \le r$.
Then
\begin{eqnarray*}
  \|x_{r+1} - x_r\| &=& \dist(P_{i_r} \cdots P_{i_1}x,M_{i_{r+1}}) \\
   &\le& \|P_{i_r} \cdots P_{i_1}x - x\| + \dist(x,M_{i_{r+1}})\\
   &\le& \|x_s - x_{s-1}\|+\|x_{s-1} - x_{s-2}\|+ \cdots + \|x_1-x\|\\
   &+& \dist(x,M_{r+1}) \\
   &\le& (2^{r-1}+2^{r-2}+\cdots+2+1+1)(\ell+\ep) = 2^r(\ell+\ep).
\end{eqnarray*}
Therefore \eqref{eq:41} holds for every $s$, and we obtain
\begin{eqnarray*}
  \|P_{i_k}P_{i_{k-1}}\cdots P_{i_1}x - x\| &=& \|x_k - x\| \\
   &\le& \|x_k - x_{k-1}\| + \|x_{k-1} - x_{k-2}\| + \cdots +\|x_1 - x\| \\
   &\le & (2^{k-1}+2^{k-2}+\cdots+1)(\ell+\ep) =  (2^{k}-1)(\ell+\ep).
\end{eqnarray*}
Thus $\gamma(P_{i_k}P_{i_{k-1}}\cdots P_{i_1}-I) \le (2^{k}-1)(\ell+\ep)$. Making $\ep \to 0$ we obtain
$$ \gamma(P_{i_k}P_{i_{k-1}}\cdots P_{i_1}-I) \le (2^{k}-1)\ell .$$
This shows that if $\ell = 0$ or, equivalently, if $c(M_1,\cdots,M_N) = 1$,  then the range of $P_{i_k}P_{i_{k-1}}\cdots P_{i_1}-I$ is not closed.

The implication ''(1$^\prime$) $\Rightarrow (1)$`` is clear. Note also that the above proof for $k = N$ and $i_s = s$ implies that
\beq\label{eq:qua}
    \frac{1}{2N^2}\ell^2 \le \gamma(T-I) \le  (2^{N}-1)\ell .
\eeq
Here $\ell = \ell(M_1,\cdots,M_N)$.

''(1$^\prime$) $\Rightarrow$ (6$^\prime$)`` \quad  Note that $\|P_{i_k} \cdots P_{i_1}-P_M\| \le 1$ always. Suppose now that $a := \|P_{i_k} \cdots P_{i_1}-P_M\| < 1$. We want to show that the range of $I-P_{i_k} \cdots P_{i_1}$ is closed. Notice first that $\Ker(I-P_{i_k} \cdots P_{i_1}) = M$ since $\{i_1, \dots, i_k\} = \{1,2,\ldots, N\}$. Let $x \in H$ be such that $\dist(x,M) = \|x-P_Mx\| = 1$.
We have
\begin{eqnarray*}
  \|(I-P_{i_k} \cdots P_{i_1})x\| &=& \|x-P_Mx +P_Mx-P_{i_k} \cdots P_{i_1}x\|\\
   &\ge& 1 - \|P_{i_k} \cdots P_{i_1}x-P_Mx\| \\
   &=& 1 - \|(P_{i_k} \cdots P_{i_1}-P_M)(x-P_Mx)\| \\
   &\ge& 1- a.
\end{eqnarray*}
Therefore the reduced minimum modulus of $I-P_{i_k} \cdots P_{i_1})$ verifies $\gamma(I-P_{i_k} \cdots P_{i_1}) \ge 1- \|P_{i_k} \cdots P_{i_1}-P_M\|$. In particular, $\Ran(I-P_{i_k} \cdots P_{i_1})$ is closed if $a<1$.

The implication ''(6$^\prime$) $\Rightarrow (6)$`` is easy.

\begin{lemma}\label{lemma:4}
Let $x\in H$, and set $u_j = P_{j} \cdots P_1x-P_Mx$ for $j\ge 1$, $u_0 = x-P_Mx$. For every $j$ with $1 \le j\le N$ we have
$$  \|u_{j-1}-u_j\|^2 \le \|u_{j-1}\|^2 - \|Tx-P_Mx\|^2
   \le \|x-P_Mx\|^2 - \|Tx-P_Mx\|^2 .$$
\end{lemma}
\begin{proof}
We can write
$$\|Tx-P_Mx\| = \|u_N\| = \|P_Nu_{N-1}\| \le \|u_{N-1}\| \le \dots \le \|u_0\| = \|x-P_Mx\| .$$
Therefore
\begin{eqnarray*}
  \|u_{j-1}-u_j\|^2 + \|Tx-P_Mx\|^2 &=& \|u_{j-1}-P_ju_{j-1}\|^2 + \|P_N\cdots P_{j+1}P_ju_{j-1}\|^2 \\
   &\le& \|u_{j-1}-P_ju_{j-1}\|^2 + \|P_ju_{j-1}\|^2 \\
   &=& \|u_{j-1}\|^2 \\
   &=& \|P_{j-1}\cdots P_1(x-P_Mx)\|^2 \\
   &\le& \|x-P_Mx\|^2 ,
\end{eqnarray*}
completing the proof of the Lemma.
\end{proof}
We continue the proof of Theorem \ref{thm:qualASC}.

''$(6) \Rightarrow (3)$`` \quad Let $j$ between $1$ and $N$. Using the Cauchy-Schwarz inequality and Lemma \ref{lemma:4} we obtain
\begin{eqnarray*}
  \dist(x,M_j)^2 &\le& \|x - P_{j} \cdots P_1x\|^2 \\
   &\le& \left(\|x - P_1x\| + \|P_1x - P_2P_1x\| + \cdots + \|P_{j-1}\cdots P_1x -P_{j}\cdots P_1x\|\right)^2\\
   &\le& j\left( \|u_{0}-u_1\|^2 + \|u_{1}-u_2\|^2 + \cdots + \|u_{j-1}-u_j\|^2\right) \\
   &\le& j^2\left( \|x-P_Mx\|^2- \|Tx-P_Mx\|^2\right) \\
   &\le& N^2\left( \|x-P_Mx\|^2- \|Tx-P_Mx\|^2\right).
\end{eqnarray*}
We get
$$ N^2\left( \|x-P_Mx\|^2- \|Tx-P_Mx\|^2\right) \ge \max_{1\le j \le N}\dist(x,M_j)^2 \ge \ell^2\|x-P_Mx\|^2,$$
which yields
$$ \|Tx-P_Mx\|^2 \le \left( 1 - \frac{\ell^2}{N^2}\right)\|x-P_Mx\|^2 .$$
In particular
\beq\label{eq:norm}
\|T-P_M\| \le \sqrt{1 - \frac{\ell^2}{N^2}} .
\eeq
Therefore $\|T-P_M\| = 1$ implies $\ell = 0$, i.e., (6) implies (3).

''(1) $\Rightarrow$ (9)`` \quad  Let $\ep > 0$. Let $K \subset M^{\perp}$ be a closed subspace of finite codimension in $M^{\perp}$. With respect to the decomposition $H = M \oplus M^{\perp}$, the operator $T$ has the following matrix decomposition
$$
T = \left(
\begin{array}{cc}
I & 0 \\
0 & A \\
\end{array}
\right) .
$$
Since $\Ran(T-I)$ is not closed, the range of the operator $I-A$, acting on $M^{\perp}$, is not closed. This means that $I-A \in \mathcal{B}(M^{\perp})$ is not an upper semi-Fredholm operator, and therefore there exists $x \in K$ such that $\|x\| = 1$ and $\|x-Ax\| \le \ep$. It follows that $\|x-Tx\| \le \ep$.

''(9) $\Rightarrow$ (4)`` \quad Let $x$ be as in (9). Then $x\in K$, $\|x\| = 1$, and $\|x-Tx\| \le \ep$. We have
$$1=\|x\| \ge \|P_1x\| \ge \|P_2P_1x\| \ge \ldots \ge \|Tx\| \ge 1-\ep .$$
Set $x_s = P_sP_{s-1}\cdots P_1x$ for $s\ge 1$ and $x_0 = x$. Then $x_s \in M_s \cap M^{\perp}$ for each $s\ge 0$ and $x_{s-1}-x_s = (I-P_s)x_{s-1}$ is orthogonal to $x_s$. Hence
$$\|x_{s-1}-x_s\|^2 = \|x_{s-1}\|^2 - \|x_s\|^2 \le 1-(1-\ep)^2 < 2\ep $$
and $\|x_{s-1} - x_s\| \le \sqrt{2\ep}$, for each $s$.
We obtain
$$ \dist(x, M_1) = \|x-P_1x\| = \|x_{0} - x_1\| \le \sqrt{2\ep}.$$
For $s\ge 1$ we have $ \dist(x, M_s) \le \|x-P_sP_{s-1}\cdots P_1x\|$; hence
$$ \dist(x, M_s) \le \|x-P_1x\|+ \|P_1x-P_2P_1x\| + \cdots + \|P_{s-1}\cdots P_1x-P_sP_{s-1}\cdots P_1x\| \le s\sqrt{2\ep}.$$
Therefore
$\max\{\dist(x,M_j) : j=1, \cdots, N\} \le N\sqrt{2\ep}.$
As $\ep>0$ is arbitrary, the proof of this implication is over.

''(4) $\Rightarrow$ (9$^\prime$)`` \quad Suppose that (4) holds. Let $\ep > 0$ and let $K \subset M^{\perp}$ be a closed subspace of finite codimension in $M^{\perp}$. Then there exists $x \in K$ such that $\dist(x,M) = \|x\| = 1$ and $\max\{\dist(x,M_j) : j=1, \cdots, N\} \le \ep$. Let $i_1, \cdots, i_k \in \{1,2, \ldots, N\}$. Set $x_0 = x$, $x_s = P_{i_s} \cdots P_{i_1}x$ for $s\ge 1$. Then $x_0 \in K$ and $x_s\in M^{\perp}\cap M_{i_s}$ for $s\ge 1$.

We shall prove by induction the following two claims :
\begin{equation}\label{eq*}
\tag{*} \dist(x_s,M_j) \le 2^s\ep \quad (j\ge 1)
\end{equation}
and
\begin{equation}\label{eq**}
\tag{**} \|x_s - x_0\| \le (2^s-1)\ep \quad (s\ge 1).
\end{equation}
Both claims are clearly true for $s=0$. Suppose that both inequalities are true for some $s\ge 0$. Then, using several times the induction hypothesis, we have
\begin{eqnarray*}
  \|x_{s+1} - x_0\| &\le& \|x_{s+1} - x_s\| + \|x_{s} - x_0\| \\
   &=& \dist(x_s,M_{i_{s+1}})+ \|x_{s} - x_0\|\\
   &\le& 2^s\ep + (2^{s}-1)\ep = (2^{s+1}-1)\ep .
\end{eqnarray*}
For $j\ge 1$ we can write
$$ \dist(x_{s+1}, M_j) \le \|x_{s+1}-x_0\|+\dist(x_0,M_j) \le (2^{s+1}-1)\ep + \ep = 2^{s+1}\ep .$$
Thus both \eqref{eq*} and \eqref{eq**} are true ; in particular we have
$$ \|P_{i_k}\cdots P_{i_1}x-x\| = \|x_k - x_0\| \le 2^k\ep .$$
As $\ep > 0$ was arbitrary, we obtain (9$^\prime$).

''(9$^\prime$) $\Rightarrow$ (8$^\prime$)`` \quad We have $P_sP_{s-1}\cdots P_1-P_M = P_sP_{s-1}\cdots P_1(I-P_M)$, so the range of this operator is in $M^{\perp}$. The assertion (9$^\prime$) implies that $1$ belongs to the essential spectrum of the restriction of $P_sP_{s-1}\cdots P_1-P_M$ to $M^{\perp}$. Therefore $1\in \sigma_e(P_sP_{s-1}\cdots P_1-P_M)$.

The implication ''(8$^\prime$) $\Rightarrow$ (8)`` is clear.

''(8) $\Rightarrow$ (7) $\Rightarrow$ (6)`` \quad The statement (8) implies the following sequence of inequalities for the essential spectral radius  $r_e(T-P_M)$ and the essential norm of $T-P_M$:
$$1 \le r_e(T-P_M) \le \|T-P_M\|_e \le \|T-P_M\| = \|P_N\cdots P_1(I-P_M)\| \le 1 .$$
Thus all inequalities are equalities.

The proofs of implications ''(8$^\prime$) $\Rightarrow$ (7$^\prime$) $\Rightarrow$ (6$^\prime$)`` are similar. The implications ''(8$^\prime$) $\Rightarrow$ (5$^\prime$) $\Rightarrow$ (5)`` are clear.

''(9$^\prime$) $\Rightarrow$ (2$^\prime$)`` \quad Let $A_k$ be the operator $P_{i_k}P_{i_{k-1}}\cdots P_{i_1}-P_M$ restricted to $M^{\perp}$. The condition (9$^\prime$) implies that $1$ is in the boundary of the essential spectrum of the operator $A_k$. According to \cite{amemiya/ando}, on the space $M^{\perp}$ the operators $A_k$ converge weakly to $0$. The assertion (2$^\prime$) can be proved exactly as in \cite[Theorem 1]{badea/muller:topol} by replacing there $T^n$ by $A_n$.

The implication ''(2$^\prime$) $\Rightarrow$ (2)`` is clear.

''(10) $\Leftrightarrow$ (3)`` \quad We have $\ell(M_1, \cdots, M_N) = 0$ if and only if $c(\mbox{ diag}(M_1),M_2\oplus \cdots \oplus M_N) = 1$.
The proof of this assertion is analogous to that of the first part of Proposition \ref{prop:CandL}. Therefore $\ell(M_1, \cdots, M_N) = 0$, or equivalently $c(M_1, \cdots, M_N) = 1$, if and only if $\mbox{ diag}(M_1) + M_2\oplus \cdots \oplus M_N$ is not closed in $H^{N-1}$.

The implication ''(11) $\Leftrightarrow$ (3)`` follows from \cite{bauschke/borwein/lewis}. The proof is complete.
\end{proof}

\subsection{Quantitative statements} Some remarks concerning the proof of Theorem \ref{thm:qualASC} are in order.

\begin{remark}The proof of Theorem \ref{thm:qualASC} gives some quantitative information between the parameters involved. Some other estimates can be proved in a similar way. For instance, we present here the quantitative version of the implication ''(6$^\prime$) $\Rightarrow$ (3)``. Let $k \ge 1$. Suppose $i_1, \dots, i_k \in \{1,2,\ldots, N\}$ and $\{i_1, \dots, i_k\} = \{1,2,\ldots, N\}$. Denote $\ell = \ell(M_1, \dots , M_N)$ and
$$a := \|P_{i_k} \cdots P_{i_1}-P_M\| \le 1 .$$
%We have $a \le 1$.
Let $\ep > 0$. There exists $x\in H$ with $\|x\| = 1$ such that $\|P_{i_k} \cdots P_{i_1}x-P_Mx\| > a-\ep$. Denote $y = x-P_Mx$ and
$$x_s = P_{i_s} \cdots P_{i_1}x-P_Mx, \quad x_0 = x-P_Mx = y  \quad (s \ge 1).$$
Clearly
$$ 1 = \|x\| = \sqrt{\|y\|^2+\|P_Mx\|^2} \ge \|y\| = \|x_0\| \ge \|x_1\| \ge \dots \ge \|x_k\| > a - \ep .$$
Since $x_{s-1} - x_s = (I-P_{i_s})x_{s-1}$ is orthogonal to $M_{i_s}$, and $x_s \in M_{i_s}$, we have
$$ \| x_{s-1} - x_s\|^2 = \| x_{s-1}\|^2 - \|x_s\|^2 \le \|y\|^2 - (a-\ep)^2 .$$
For each $r \in \{1, \dots , k\}$ we have
\begin{eqnarray*}
  \|x - P_Mx - x_r\| &\le& \|x_0 - x_1\| + \|x_1 - x_2\| + \cdots + \|x_{r-1} - x_r\|\\
   &\le& k\sqrt{\|y\|^2 - (a-\ep)^2} .
\end{eqnarray*}
Since $\{i_1, \dots, i_k\} = \{1,2,\ldots, N\}$, for each $j \in \{1, 2, \ldots, N\}$ we have $\{x_1, \dots, x_k\} \cap M_j \neq \emptyset$. Therefore
$$ \dist(x,M_j) \le \max \{\|x-P_Mx-x_r\| : 1\le r \le k\} \le k\sqrt{\|y\|^2 - (a-\ep)^2}$$
and
$$ \|y\|^2\ell^2 \le \max \{ \dist(x,M_j)^2 : 1\le j \le k\} \le k^2\left(\|y\|^2 - (a-\ep)^2\right) .$$
Hence $k^2(a-\ep)^2 \le (k^2-\ell^2)\|y\|^2 \le (k^2-\ell^2)$.
As this is satisfied for every $\ep > 0$, we obtain $ \|P_{i_k} \cdots P_{i_1}-P_M\| = a \le (1 - \ell^2/k^2)^{1/2}. $
\end{remark}

\begin{theorem}\label{cor:main}
Let $N \ge 2$. Let $H$ be a complex Hilbert space. Let $M_1, \cdots, M_N$ be $N$ closed subspaces of $H$ with intersection $M = M_1\cap M_2 \cdots \cap M_N$. Let $P_j= P_{M_j}$, $1\le j \le N$, and $P_M$ be the corresponding orthogonal projections. Denote $T = P_N\cdots P_1$.

(i) \quad Suppose that $c := c(M_1, \cdots, M_N) < 1$. Then $(T^n)_{n\ge 1}$ is uniformly convergent to $P_M$, with
$$ \|T^n - P_M\| \le \left(1-\left(\frac{1-c}{4N}\right)^2\right)^{n/2} \quad (n \ge 1).$$

(ii) \quad Suppose that $c := c(M_1, \cdots, M_N) = 1$. Then $(T^n)_{n\ge 1}$ is strongly convergent to $P_M$ and we have (ASC), in all possible meanings of this paper.
\end{theorem}

\begin{proof} We have to prove only the estimate in Part $(i)$. Suppose that $c := c(M_1, \cdots, M_N) < 1$.
Denote $\widetilde{M_j} = M_j\cap M^{\perp}$ and $Q_j = P_{\widetilde{M_j}}$ for $1\le j \le N$. Then the intersection of $\widetilde{M_j}$,  $1\le j \le N$, is $\{0\}$ and, according to Proposition \ref{prop:FRconf}, (b), we have $c(\widetilde{M_1},\ldots, \widetilde{M_N}) = c < 1$. We also have $T^n - P_M = (Q_NQ_{N-1}\cdots Q_1)^n$ for each $n\ge 1$ (see \cite[Lemma 9.30]{deutsch:book}). We apply (\ref{eq:norm}), which was proved in the implication ''(6) $\Rightarrow$ (3)`` of Theorem \ref{thm:qualASC}, to $Q_j$. We obtain
$ \|Q_NQ_{N-1}\cdots Q_1\| \le \sqrt{1 - \frac{\ell^2}{N^2}} ,$
where now $\ell = \ell(\widetilde{M_1},\ldots, \widetilde{M_N})$. According to Proposition \ref{prop:CandL}, applied for the subspaces $\widetilde{M_j}$, we have $1-\frac{2N}{N-1}\ell \le c$. This implies
$\ell^2 \ge \left(\frac{N-1}{2N}\right)^2(1-c)^2 \ge \frac{1}{16}(1-c)^2.$
Therefore
\begin{eqnarray*}\|T^n - P_M\|^2 &=& \|(Q_NQ_{N-1}\cdots Q_1)^n\|^2 \le \|(Q_NQ_{N-1}\cdots Q_1)\|^{2n}\\
 &\le& \left(1 - \frac{\ell^2}{N^2}\right)^n \le \left(1-\left(\frac{(1-c)^2}{16N^2}\right)\right)^{n} ,
\end{eqnarray*}
which implies $(i)$.
\end{proof}

\subsection{Comparison with other estimates}
Let $M_{1},\ldots, M_{N}$ be $N$ closed subspaces of $H$,
with intersection $M = M_1\cap \cdots \cap M_N$. Denote $c_{ij} =
c_0(M_{i}\cap M^{\perp},M_{j}\cap M^{\perp})$ for $1 \le i,j \le
N$.

It was proved in \cite[Theorem 2.1]{DeHu} that
$$\|(P_{N}\ldots P_{1})^{n}-P_{M}\|\leq c_{1N}^{n-1}c_{12}^{n}\ldots
c_{N-1,N}^{n}.$$ In particular, we have quick uniform convergence
whenever one of the cosine $c_{i,i+1}$ of the Dixmier angles is
strictly less than one.

Moreover, for any sequence $i_{1},\ldots i_{N}$ of integers such
that $\{i_{1},\ldots, i_{N}\}=\{1,\ldots, N\}$, Theorem
\ref{thm:qualASC} shows that we have (QUC) for $(P_{N}\ldots
P_{1})^{n}$ if and only if we have (QUC) for $(P_{i_{N}}\ldots
P_{i_{1}})^{n}$. Hence we have (QUC) for $T^{n}=(P_{N}\ldots
P_{1})^{n}$ as soon as there exist integers $i\not = j$ such that
$c_{ij} = c_0(M_{i}\cap M^{\perp},M_{j}\cap M^{\perp})<1$.  The
following example shows that this sufficient condition for (QUC)
is by far stronger than the condition $c(M_{1},\ldots, M_{N})<1$.

\begin{example}
Let $(e_{n})_{n\geq 0}$ be an orthonormal basis of $H$, and let
$M_{1}$, $M_{2}$ and $M_{3}$ be the following closed subspaces of
$H$: $M_{1}=\overline{\textrm{span}}[e_{3n} \textrm{ ; } n\geq
0]$, $M_{2}=\overline{\textrm{span}}[e_{0},\, e_{3n+1} \textrm{ ;
} n\geq 0]$ and $M_{3}=\overline{\textrm{span}}[e_{1},\, e_{3},\,
e_{3n+2} \textrm{ ; } n\geq 0]$. Then $M_{1}\cap
M_{2}=\textrm{span}[e_{0}]$, $M_{2}\cap
M_{3}=\textrm{span}[e_{1}]$, $M_{1}\cap
M_{3}=\textrm{span}[e_{3}]$ and $M_{1}\cap M_{2}\cap M_{3}=\{0\}$.
We obtain $c_{ij} = c_0(M_{i}, M_{j})=1$ for any $i$ and $j$. But
$c(M_{1},M_{2},M_{3})<1$: indeed if $x=\sum_{n\geq 0}x_{n}e_{n}$,
then a straightforward computation shows that
$||\frac{1}{3}(P_{1}+P_{2}+P_{3})||=\frac{2}{3}$, so that
$c(M_{1},M_{2},M_{3})=\frac{1}{2}$.
\end{example}

The following proposition shows, even in a quantitative way, that the sufficient condition $c(M_1\cap \cdots \cap M_{j-1}, M_j) < 1$, for each $j$, reminiscent of \cite[Theorem 2.2]{SSW} and \cite[Theorem 2.7]{DeHu}, implies that $c(M_1, \dots , M_N) < 1$.

\begin{proposition}\label{prop:estimC}
Let $N \ge 2$. Let $M_1, \cdots, M_N$ be $N$ closed subspaces of $H$ with intersection $M = M_1\cap \cdots \cap M_N$. Denote $c_j = c(M_1\cap \cdots \cap M_{j-1}, M_j)$ for $j$ between $2$ and $N$. Then
$$c(M_1, \dots , M_N) \le 1 - \frac{1}{N-1}\prod_{j=2}^N\left(1 - \sqrt{\frac{c_j+1}{2}}\right)^2 \le 1 - \frac{1}{(N-1)4^{N-1}}\prod_{j=2}^N\left(1-c_j\right)^2 .$$
In particular, $c(M_1, \dots , M_N) < 1$ if each $c_j < 1$, $2\le j \le N$.
\end{proposition}
\begin{proof}
The estimates are clear if one of the $c_j$'s is one. Suppose $c_j < 1$ for every $j$. Denote $\ell_j = c(M_1\cap \cdots \cap M_{j-1}, M_j) > 0$ for $j$ between $2$ and $N$. It follows from the proof of \cite[Theorem 5.11]{bauschke/borwein} that $\ell(M_1, \dots , M_N) \ge \ell_2\ell_3\cdots \ell_N$. The proof of Proposition \ref{prop:CandL} for $N=2$, and two given subspaces $S_1$ and $S_2$, yields
$$1-\ell(S_1,S_2) \le \sqrt{\kappa(S_1,S_2)} = \sqrt{\frac{c(S_1,S_2)+1}{2}}.$$
This implies that
$$\ell(S_1,S_2) \ge 1 - \sqrt{\frac{c(S_1,S_2)+1}{2}} \ge \frac{1-c(S_1,S_2)}{4} .$$
Using Proposition \ref{prop:CandL} we obtain
\begin{eqnarray*}
c(M_1, \dots , M_N) &\le& 1 - \frac{\ell(M_1, \dots , M_N)^2}{N-1} \\
   &\le& 1 - \frac{\ell_2^2\ell_3^2\cdots \ell_N^2}{N-1}
   \end{eqnarray*}
   \begin{eqnarray*}
  &\le& 1 - \frac{1}{N-1}\prod_{j=2}^N\left(1 - \sqrt{\frac{c_j+1}{2}}\right)^2 \\
 &\le& 1 - \frac{1}{(N-1)4^{N-1}}\prod_{j=2}^N\left(1-c_j\right)^2 ,
\end{eqnarray*}
which completes the proof.
\end{proof}

%\end{document}

\end{document}